\documentclass[]{article}

\usepackage{mathptmx}
\usepackage{amsmath}
\usepackage{amssymb}
\usepackage{amsthm}  
\usepackage{wasysym}  
\usepackage{bm}		
\usepackage{doi}	
\usepackage{fullpage}			
\usepackage{graphicx} 
\usepackage{appendix}

\newcommand{\norm}[1]{\left\lVert{#1}\right\rVert}

\newtheorem{assumption}{Assumption}
\newtheorem{lemma}{Lemma}
\newtheorem{definition}{Definition}
\newtheorem{proposition}{Proposition}

\newcommand{\citep}[1]{\cite{#1}}

\title{\Large A Continuous Feedback Optimal Control based on Second-Variations for Problems with Control Constraints}

\author{Joris T. Olympio}

\date{}

\begin{document}

\maketitle

\begin{abstract}
The paper describes a continuous second-variation algorithm to solve optimal control problems where the control is defined on a closed set. A second order expansion of a Lagrangian provides linear updates of the control to construct a locally feedback optimal control of the problem. Since the process involves a backward and a forward stage, which require storing trajectories, a method has been devised to accurately store continuous solutions of ordinary differential equations. Thanks to the continuous approach, the method adapts implicitly the numerical time mesh. The novel method is demonstrated on bang-bang optimal control problems, showing the suitability of the method to identify automatically optimal switching points in the control.
\end{abstract}


\section{Introduction}
The current paper is concerned with the resolution of optimal control problems with bang-bang structure. 
Most optimal control problem solvers use a direct problem formulation\citep{Betts} that transforms the optimal control problem into a nonlinear program\citep{ross_dido} \citep{rao_gops}\citep{becerra_psopt}\citep{betts_socs}. Indeed, one could argue that direct methods offer the most straightforward formulation to a generic optimal control solver, and they are quite robust. Alternative methods, for instance indirect methods\cite{Pontryaguin}, require formulating a boundary value problem and convergence is usually difficult. Neither direct, nor indirect method-based algorithms can be both robust and accurate in computing the optimal control solution. Besides, direct and indirect approach, one should explore second-order methods.

Second-order methods appear as very good alternative. In the 1960's, while trying to elegantly solve LQP researchers devised second-order methods such as second-variation methods or differential dynamic programming (DDP)\citep{mcreynolds} \citep{jacobson} \cite{mayne} \citep{mitter} \citep{Gershwin} \citep{dyer}. The classical second order algorithms for optimal control problems use dynamic programming equations. Referred as differential dynamic programming (DDP), DDP provides a simple way of computing an optimal control through a quadratic expansion of the Hamilton-Jacobi-Bellman value function around a nominal trajectory\citep{jacobson}. Over the past 50 years, these methods have also been tested and adapted on non-LQP problems, and in the space community, DDP based methods seem to get a new lease of life. Current computer architecture allow better programing of the methods thanks to significant increase in memory capacity and processing unit speed. So far, Whiffen \citep{whiffen.2001}\citep{whiffen.2002} certainly proposed the best improvement of DDP method to provide a generic tool for high-fidelity low-thrust space trajectory optimisation. The major difficulty in low-thrust optimal control problems is the control is often of a bang-bang structure because the control is linear in the dynamics. 
Some work using DDP then consider simplifications in the dynamics\citep{lantoine}, or consider a quadratic cost to the designing of low-thrust trajectories to asteroids to improve convergence rate\citep{Colombo}. To the author knowledge there is not a lot of supporting theories for convergence and behaviour of DDP.\par
Second-order methods can also be devised using Lagrangian expansions, rather than the Bellman's value function. The approach has been introduced by Bullock\citep{bullock.rpt}\citep{bullock} and Bryson and  extended by others for space trajectory optimization\citep{Olympio.Phd}\citep{Olympio}. 
Although the basis is different from the derivation of DDP the end results are quite similar. Both DDP and gradient-based Lagrangian methods compute a linear control update iteratively using gradient of the Hamiltonian function. However, the derivations based on the expansion of the Lagrangian rather than the Hamilton-Jacobi-Bellman value function provide a better theoretical framework. \vspace{\baselineskip}

The algorithm presented in this paper uses as baseline an expansion of the Lagrangian. The novelty of the approach is however many fold. This paper proposes an approach for approximating the continuous solution through interpolation, thus improving the backward-forward integrations, and increasing the accuracy of the sensitivity matrices and the feedback coefficients. A method for handling dynamical control constraints without using penalization is presented. Penalization of the control or the state for constraints is one of the major drawback of current DDP and second-variation methods. That is, a regularisation of the Lagrangian is also presented to improve the convergence rate.

In section~\ref{sec:problem}, the main problem is introduced with notations and assumptions. 
The methodology is developed in section~\ref{sec:sogo}. This includes the second-order developments, the generation of an optimal feedback control law, the conditions of optimality, and the treatment of state constraints.
Specific section~\ref{sec:Control} presents the optimal control update, and deals with the particular case of control constraints, which is important for the bang-bang problem class. 
Section~\ref{sec:mesh} introduced a practical solution method to produce continuous solution for the problem, with an adaptive implicit mesh.
The algorithm of the method is presented in section~\ref{sec:algo}.
Eventually, the last section, section~\ref{sec:examples} illustrates the method, and demonstrates that the method can find the bang-bang optimal control without requiring any insights on the switching structure.

\section{Problem Statement} \label{sec:problem}

\subsection{Preliminaries}
To simplify notation, the following conventions are used. Scalars are denoted in italics, while column vectors, matrices and tensors are in bold. In addition, column vectors use lowercase and matrices use uppercase.
\par 
And for derivation, 
\begin{equation*}
\mathbf{F_{\mathbf{x}}} = \frac{\partial F}{\partial \mathbf{x}} \qquad \mathbf{F_{xy}} = \frac{\partial^2 F}{\partial \mathbf{x} \partial \mathbf{y}} 
\end{equation*}
where $F$ is a general scalar function. For clarity of the expressions, when the points of evaluation are omitted it is assumed that the expression is evaluated around a nominal trajectory $[\bar{\mathbf{x}}(t), \bar{\mathbf{u}}(t), \bar{p}]$. The dimension of the state $\mathbf{x}(t)$ is $n_x$, and the dimension of control $\mathbf{u}(t)$ is $n_u$.


\subsection{Formulation}

The developments include the case where a scalar constant parameter $p$ is included and needs to be optimised along with the control.\par
The dynamics are described by an ordinary differential equation
\begin{align}
    \frac{d \mathbf{x}}{dt} &= \mathbf{f}(\mathbf{x}, \mathbf{u}, p; t), \quad t \in [t_0, t_f], \quad \mathbf{u}(t) \in U, \quad p \in P
	\label{eq:dynamic} \\
	0 &= \bm{\varphi}(\mathbf{x}(t_0), \mathbf{x}_0,p_0) \label{eq:initial_conditions}
\end{align}
where $\bm{\varphi}:\mathbb{R}^{n_x} \times \mathbb{R} \rightarrow \mathbb{R}^{n_{\varphi}}$ describes initial conditions on the state at date $t_0$ and unknown parameter $p_0$. Functional $\mathbf{f}$, of state function $\mathbf{x} \in C([t_0, t_f], \mathbb{R}^{n_x})$ and unknown constant parameter $p \in P \subset \mathbb{R}$, is assumed at least twice continuously differentiable, $\mathbf{f} \in W^{2}_{2}(\mathbf{R}^{n_x} \times U \times P)$, with the Sobolev space
\begin{equation*}
	W^{k}_{p}(\Omega) = \{ f \in L^2(\Omega): \forall \alpha \in \mathbb{N}, |\alpha| \leq k, \partial^{\alpha} f \in L^p(\Omega)\} 
\end{equation*}

Furthermore, the control is assumed a bounded measurable function, $\mathbf{u} \in L^2([t_0, t_f], U)$. Let the control region $U \subset \mathbb{R}^{n_u}$ be defined by the box constraints:
\begin{equation}
	u_i^l \leq u_i \leq u_i^u		\quad 1 \leq i \leq n_u
	\label{eq:control_box_constraint}
\end{equation}
where $u_i^l \in \mathbb{R}$ and $u_i^u \in \mathbb{R}$ are respectively the lower and upper bounds on the $i^{th}$ component of the control. 

Let terminal constraints be defined as
\begin{equation}
    \bm{\psi}(\mathbf{x}(t_f), p_f) = 0
	\label{eq:terminal_constraint}
\end{equation}
where $\bm{\psi}:\mathbb{R}^{n_x} \times \mathbb{R} \rightarrow \mathbb{R}^{n_{\psi}}$ is assumed twice continuously differentiable. Scalar $p_f$ is an unknown parameter. With very few modifications, these problem parameters can also be slack variables of possible terminal inequality constraints.

Consider the minimisation of an objective function $J$ written under the Mayer form,
\begin{equation}
    J = J(\mathbf{x}(t_f), p_f) \longrightarrow \min
	\label{eq:objective_function}
\end{equation}
where $J:\mathbb{R}^{n_x} \times \mathbb{R} \rightarrow \mathbb{R}$ is a functional of the vector $\mathbf{x}(t_f)$ and the parameter $p$. It is assumed well defined (Jacobian of maximum rank), and at least twice continuously differentiable. \par

A second order method is presented to minimise the objective function $J$, Eq.~(\ref{eq:objective_function}), subjects to the state constraints Eqs.~(\ref{eq:initial_conditions}) and (\ref{eq:terminal_constraint}), and the control constraint Eq.~(\ref{eq:control_box_constraint}), with iterative updates on the control $\mathbf{u}(t)$ and the parameters $p$,$p_0$ and $p_f$.

\section{Second Order Gradient Method} \label{sec:sogo}
\subsection{Second Order Equations}


For this problem, an extended terminal value function is introduced,
\begin{equation}
    \phi(\mathbf{x}, \bm{\nu}, p_f; t_f) =  J(\mathbf{x}(t_f), p_f) + \bm{\nu}^T \bm{\psi}(\mathbf{x}(t_f), p_f) + \bm{\psi}(\mathbf{x}(t_f), p_f)^T \mathbf{C} \bm{\psi}(\mathbf{x}(t_f), p_f)
\label{eq:value_function}
\end{equation}
where $\bm{\nu} \in \mathbb{R}^{n_{\psi}}$ is a non-zero constant Lagrange multiplier vector for the constraints $\bm{\psi}$, and $\mathbf{C} > 0$ is a square diagonal regularisation matrix, $\mathbf{C} \in M^{n_{\psi}, n_{\psi}}(\mathbb{R})$. 
Let the augmented Lagrangian $\mathcal{L} \in W^{2}_{2}(\mathbb{R}^{nx} \times U \times P \times \mathbf{R}^2 \times \mathbb{R}^{n_{\psi}} )$ of the constrained optimal control problem be defined with 
\begin{equation}
	\mathcal{L}(\mathbf{x}, \mathbf{u}, p, p_0, p_f, \bm{\nu}) = \phi(\mathbf{x}, \bm{\nu}, p_f; t_f) + \bm{\eta}^T \bm{\varphi}(\mathbf{x}_0, p_0; t_0) + \int_{t_0}^{t_f}{ \bm{\lambda}^T \left( \mathbf{f}(\mathbf{x}, \mathbf{u}, p; t) - \frac{d \mathbf{x}}{dt} \right) dt}
\label{eq:augmented_lagrangian}
\end{equation}
It accounts for initial conditions Eq.~(\ref{eq:initial_conditions}), terminal constraints Eq.~(\ref{eq:terminal_constraint}) and objective function Eq.~(\ref{eq:objective_function}). Costate vector $\bm{\lambda}$ and Lagrange multiplier $\bm{\nu}$ have been introduced for the dynamics and the initial constraints Eq.~(\ref{eq:initial_conditions}) respectively.

While most second-order methods published to date only use penalisation to reach constraint satisfaction, the presented method introduces Lagrange multipliers to have a quicker and better satisfaction of the constraints. As part of a min-max problem, $V$ should be minimised for $\mathbf{u}$ and $p$, while maximised for $\bm{\nu}$. Indeed, the constrained maximisation supposes the existence of a min-max and max-min problem, or also the existence of a saddle point. The regularisation matrix $\mathbf{C}$ of the augmented Lagrangian (Eq.~(\ref{eq:augmented_lagrangian})) is thus used to regularise the Lagrangian and find a saddle point of $\mathbf{L}$ by reducing any duality gap.
Gill\cite{gill.technote}, Hestenes\cite{hestenes}, Powell\cite{powell} and others have demonstrated that solving the extended function of the unconstrained problem Eq.~(\ref{eq:value_function}) is equivalent to solving the constrained problem Eqs.(\ref{eq:terminal_constraint} - \ref{eq:objective_function}).\par 
\vspace{\baselineskip}

Let the Hamiltonian $H \in W^{2}_{2}(\mathbb{R}^{nx} \times \mathbb{R}^{nx} \times U \times \mathbb{R} )$ be defined as:
\begin{equation}
	H(\mathbf{x}, \bm{\lambda}, \mathbf{u}, p; t) = \bm{\lambda}(t)^T \mathbf{f}(\mathbf{x}, \mathbf{u}, p; t)
\end{equation}
since the objective function $J$ (Eq.~\ref{eq:objective_function}) does not include Lagrangian terms (integral terms).

Looking for control updates $\delta \mathbf{u}$, the unknowns of the problem are $\delta \mathbf{x}$, $d p$ and $d \bm{\nu}$, such that a control update can be expressed in the form
\begin{align}
\delta \mathbf{u} &= \bm{\alpha} + \bm{\beta} \delta \mathbf{x} + \bm{\gamma} dp + \bm{\omega} d \bm{\nu} \label{eq:feedback_control} \\
	\mathbf{u} &= \mathbf{\bar{u}} + \epsilon_u \delta \mathbf{u} \label{eq:control_update}
\end{align}
with $\bm{\alpha} \in \mathbb{R}^{n_u}$, $\bm{\beta}\in M^{n_u, n_x}(\mathbb{R})$, $\bm{\gamma} \in \mathbb{R}^{n_u}$ and $\bm{\omega} \in M^{n_u, n_{\psi}}(\mathbb{R})$. $\mathbf{\bar{u}}$ is the current nominal control and $\epsilon_u \in [0, 1]$ is an update factor. This control update can be seen as a feedback control law on the state $\mathbf{x}$.\par 

Since $J$, $\mathbf{f}$, $\bm{\varphi}$ and $\bm{\psi}$ are assumed at least twice continuously differentiable, the development of the Lagrangian $\mathcal{L} \in W^{2}_{2}(\mathbb{R}^{n_x} \times U \times \mathbb{R})$ to second order, around the nominal trajectory $\bar{\mathbf{x}}$ and control $\bar{\mathbf{u}}$ and nominal parameter values,  yields, 

\begin{equation}
	\begin{split}
	d \mathcal{L}(\bar{\mathbf{x}}, \bar{\mathbf{u}}, p, p_0, p_f, \bm{\nu}) &= 
		  d \phi + d \left(\bm{\eta}^T \bm{\varphi} \right)\\
		&- \bm{\lambda}_f^T \delta \mathbf{x}_f + \bm{\lambda}_0^T \delta \mathbf{x}_0 \\
		&- \delta \bm{\lambda}_f^T \delta \mathbf{x}_f + \delta \bm{\lambda}_0^T \delta \mathbf{x}_0 \\
		&+ \int_{t_0}^{t_f}{ \left( \left( \frac{d \bm{\lambda} + \delta \bm{\lambda}}{dt} \right)^T \delta \mathbf{x} 
		- \delta \bm{\lambda}^T \frac{d \mathbf{x}}{dt} \right) dt} 
		+ \int_{t_0}^{t_f}{ dH dt}\\
		&+ o(\norm{\delta \mathbf{x}}^2, \norm{\delta \bm{\lambda}}^2, dp^2)
\end{split}
\label{eq:scddevel}
\end{equation}

with:
\begin{equation}
	\begin{split}
   dH &= \mathbf{H_x} \delta \mathbf{x} + \mathbf{H_u} \delta \mathbf{u} + H_p dp + \mathbf{f} \delta \bm{\lambda} \\
		&+ \frac{1}{2} \delta \mathbf{x}^T \mathbf{H_{xx}} \delta \mathbf{x} + \frac{1}{2} \delta \mathbf{u}^T \mathbf{H_{uu}} \delta \mathbf{u} + \frac{1}{2} H_{pp} dp^2\\
		&+ \delta \mathbf{u}^T \mathbf{H_{ux}} \delta \mathbf{x} + \delta \mathbf{x}^T \mathbf{H_{xu}} \delta \mathbf{u} + \delta \mathbf{x}^T \mathbf{H_{xp}} dp + H_{px} \delta \mathbf{x} dp\\
		&+ \delta \mathbf{u}^T \mathbf{H_{up}} dp + \mathbf{H_{pu}} \delta \mathbf{u} dp\\
		&+ \delta \bm{\lambda}^T \mathbf{f}_{x} \delta \mathbf{x} + \delta \bm{\lambda}^T \mathbf{f}_{u} \delta \mathbf{u} + \delta \bm{\lambda}^T \mathbf{f}_{p} dp\\
		&+ \delta \mathbf{x}^T \mathbf{f}_{x} \delta \bm{\lambda} + \delta \mathbf{u}^T \mathbf{f}_{u} \delta \bm{\lambda} + \mathbf{f}_{p} \delta \bm{\lambda} dp
	\label{eq:devel_hamiltonian}
	\end{split}
\end{equation}

Residual term $o(\norm{\delta \mathbf{x}}^2, \norm{\delta \bm{\lambda}}^2, dp^2)$ describes the higher order terms of the development, and,
\begin{align}
o(\norm{\delta \mathbf{x}}^2, \norm{\delta \bm{\lambda}}^2, dp^2) &= \Delta \mathcal{L}
- d\mathcal{L}(\mathbf{x}, \mathbf{u}, p, p_0, p_f, \bm{\nu}) \\
\Delta \mathcal{L} &= \mathcal{L}(\bar{\mathbf{x}}, \bar{\mathbf{u}}, \bar{p}, \bar{p_0}, \bar{p_f}, \bar{\bm{\nu}}) - \mathcal{L}(\mathbf{x}, \mathbf{u}, p, p_0, p_f, \bm{\nu}) \label{eq:DeltaLagrangian}
\end{align}

$\Delta \mathcal{L}$ is a measure of the improvement owing to the updated control law (Eq.~(\ref{eq:control_update})). The second term of the right hand side must thus be of the same order of magnitude of $\Delta \mathcal{L}$. For linear or  quadratic problems in $\mathbf{x}$, $\mathbf{u}$ and $p$, the residual term is zero.

\begin{lemma}[Riccati Transformation]\label{lemma:costate}
The costate vector is related to the other variables with the transformation,
\begin{equation}
	\delta \bm{\lambda} = \mathbf{R} \delta \mathbf{x} + \mathbf{K} d\bm{\nu} + \mathbf{T} dp 
\label{eq:trf_lambda}
\end{equation}
with $\mathbf{R} \in \mathbb{R}^{n_x, n_x}(\mathbb{R})$, $\mathbf{K} \in \mathbb{R}^{n_x, n_{\phi}}(\mathbb{R})$ and  $\mathbf{T} \in \mathbb{R}^{1, n_x}(\mathbb{R})$. 
\end{lemma}

\begin{proof}
The state transition matrix $\bm{\Phi}(t_f, t)$ for the dynamics $\mathbf{f}$ yields, 
\begin{align}
	\label{app:stm_x}
	\mathbf{x}(t_f) &= \bm{\Phi}_{11}(t_f, t) \mathbf{x}(t) + \bm{\Phi}_{12}(t_f, t) \bm{\lambda}(t) + \bm{\Phi}_{13}(t_f, t) p\\
	\label{app:stm_l}
	\bm{\lambda}(t) &= \bm{\Phi}_{21}(t, t_f) \mathbf{x}(t_f) + \bm{\Phi}_{22}(t, t_f) \bm{\lambda}(t_f) + \bm{\Phi}_{23}(t, t_f) p\\
	\label{app:stm_p}
	p &= \bm{\Phi}_{31}(t, t_f) \mathbf{x}(t_f) + \bm{\Phi}_{32}(t, t_f) \bm{\lambda}(t_f) + \bm{\Phi}_{33}(t, t_f) p
\end{align}

The transversality conditions for the terminal constraints provide the value for $\bm{\lambda}(t_f)$. The transition matrix is invertible. Combining equations~(\ref{app:stm_x}) and (\ref{app:stm_l}), and considering small variations of the variables $\mathbf{x}$,  $p$ and $\bm{\nu}$ yield:
\begin{equation*}
	\begin{split}
	\delta \bm{\lambda}(t) &= (1 - \bm{\Phi}_{21}(t, t_f) \bm{\Phi}_{12}(t_f, t))^{-1} \bm{\Phi}_{21}(t, t_f) \bm{\Phi}_{11}(t_f, t) \delta \mathbf{x}(t) \\
				&+ (1 - \bm{\Phi}_{21}(t, t_f) \bm{\Phi}_{12}(t_f, t))^{-1} (\bm{\Phi}_{21}(t, t_f) \bm{\Phi}_{13}(t_f, t) + \bm{\Phi}_{23}(t, t_f) ) \delta p \\
				&+ (1 - \bm{\Phi}_{21}(t, t_f) \bm{\Phi}_{12}(t_f, t))^{-1} \bm{\Phi}_{22}(t, t_f) \psi_{\mathbf{x}_f}^T d \bm{\nu}
	\end{split}
\end{equation*}
\end{proof}

\begin{assumption}\label{ass:solution_boundedness}
Consider a nominal trajectory $\bar{\mathbf{x}}$ and a nominal control $\bar{\mathbf{u}}$, possibly not optimal. The solutions of Eqs.~(\ref{eq:odeR}, \ref{eq:odeK},\ref{eq:odeT}) are bounded with the following assumptions\citep{kalman}\cite{bucy.1967}\citep{jacobson.ricatti}:
\begin{itemize}
	\item  Legendre-Clebsch condition $\mathbf{H_{uu}} > 0$. This discards any problems where the control appears linearly. A procedure that guarantees this condition in the method shall be presented hereafter.
	\item $\mathbf{H_{xx}} - \mathbf{H_{xu}} \mathbf{H_{uu}^{-1}} \mathbf{H_{ux}}$ is positive-semidefinite.
	\item $\mathbf{R}(t_f)$ is positive-semidefinite
\end{itemize}
\end{assumption}

The first assumption is a necessary condition to ensure that no singular arcs are encounter. However, this has no influence on the local optimality of the final solution.\par

\begin{definition}\label{def:trustregion}
Considering a quadratic model $(\mathbf{Q}, \mathbf{R})$, the Trust-Region problem consists in minimising the quadratic expansion for the solution to be within a given ball $B(\mathbf{x}, r)$,
\begin{equation*}
	\begin{aligned}
		\min_{\mathbf{p}}&{ \left( \mathbf{f} + \mathbf{p}^T\mathbf{R} + \frac{1}{2} \mathbf{p}^T \mathbf{Q} \mathbf{p}\right)} \\
		s.t.& \\
		&\norm{\mathbf{p}} \leq \Delta
	\end{aligned}
\end{equation*}
where $\mathbf{p}=\delta \mathbf{x}$ and $\Delta$ is called the Trust-Region radius. 
Furthermore, the solution $\mathbf{p^*}$ is solution of linear problem
\begin{equation*}
	\left( \mathbf{Q} + \lambda \mathbf{I} \right) \mathbf{p} = -\mathbf{R}
\end{equation*}
for some $\lambda \geq 0$ such that $\mathbf{Q} + \lambda \mathbf{I}$ is positive definite.
\end{definition}

Applying the Trust-Region algorithm\citep{Rodriguez97trustregion} \citep{Conn.TrustRegion}\citep{Coleman93anefficient} results then in the computation of a shift matrix $\mathbf{S_{uu}} = \lambda \mathbf{I}$ such that  the Hessian $\mathbf{H_{uu}} \leftarrow \mathbf{H_{uu}} + \lambda \mathbf{I}$ is positive-definite. A counterpart of applying the Trust-Region method to correct the Hessian is a limitation of the control update amplitude $\delta \mathbf{u}$ in accordance to the selection of Trust-Region radius $\Delta$, which eventually has an influence on result of Eq.~(\ref{eq:DeltaLagrangian}).\par

Transformation of Eq.~(\ref{eq:trf_lambda}) allows to compute the feedback coefficients in Eq.~	(\ref{eq:feedback_control}). This yields, assuming $\mathbf{H_{uu}}>0$, 
\begin{align}
\bm{\alpha} &= -\mathbf{H_{uu}}^{-1} \mathbf{H_u}\\
\bm{\beta} &=  -\mathbf{H_{uu}}^{-1} \left( \mathbf{H_{xu}} + \mathbf{R}^T \mathbf{f}_u \right)^T\\
\bm{\gamma} &=  -\mathbf{H_{uu}}^{-1} \left( \mathbf{H}_{p \mathbf{u}} + \mathbf{T}^T \mathbf{f}_u \right)^T\\
\bm{\omega} &=  -\mathbf{H_{uu}}^{-1} \left(\mathbf{K}^T \mathbf{f_u} \right)^T
\end{align}

And, using the conditions of optimality, ordinary differential equations for $\mathbf{R}$, $\mathbf{K}$, $\mathbf{T}$ and the costate vector $\bm{\lambda}$ are obtained.
Thus, 
\begin{align}
	{\frac{d \bm{\lambda}}{dt}}^T &= -\mathbf{H_x} + \mathbf{H_u} \mathbf{H_{uu}}^{-1} \left( \mathbf{H_{ux}} + \bm{H_{u\lambda}} \mathbf{R} \right)\\
	{\frac{d \mathbf{x}}{dt}} &= \mathbf{H_{\bm\lambda}}^T = \mathbf{f}\\
-\frac{d \mathbf{R}}{dt} &= \mathbf{H_{xx}} + \mathbf{R}^T \mathbf{f}_{x} + \mathbf{f}_{x} \mathbf{R} + \bm{\beta}^T \left(\mathbf{H_{ux}} + \mathbf{f}_u \mathbf{R} \right)
\label{eq:odeR}\\
-\frac{d \mathbf{K}}{dt} &= \mathbf{K}^T \mathbf{f}_{x} + \bm{\omega}^T \left( \mathbf{H_{ux}} + \mathbf{f}_u \mathbf{R} \right)
\label{eq:odeK}\\
-\frac{d \mathbf{T}}{dt} &= \mathbf{H_{px}} + \mathbf{T}^T \mathbf{f}_{x} + \bm{\gamma}^T \left( \mathbf{H_{ux}} + \mathbf{f}_u \mathbf{R} \right) + \mathbf{f}_p \mathbf{R}
\label{eq:odeT}
\end{align}

These ODEs allow of computing a bounded update $\delta \mathbf{u}$ of the current nominal control $\bar{\mathbf{u}}$ through Eq.~(\ref{eq:feedback_control}). If the truncation error defined by Eq.~(\ref{eq:DeltaLagrangian}) is limited, this control update should minimise the Lagrangian of the problem and to some extent the terminal constraints violation and/or the terminal cost $J$.\par 
Additionally, to decrease the constraint violation, an update of the Lagrange multiplier $\bm{\nu}$ should be provided, and similarly, an update for the parameter $p$ should be produced to minimize the Lagrangian. 

\subsection{Parameter and Lagrangian Vector Updates}
\begin{proposition}\label{prop:transformation}
The following transformations are considered
\begin{align}
	d \bm{\psi} &= \mathbf{K}^T \delta \mathbf{x} + \mathbf{Q} d\bm{\nu} + \mathbf{V} dp \label{eq:trf_psi}\\
	d \bm{\tau^p} &= \mathbf{T}^T \delta \mathbf{x} + \mathbf{V}^T d\bm{\nu} + \mathbf{W} dp \label{eq:trf_tau}
\end{align}
\end{proposition}
This can be demonstrated by computing the ODEs and identifying the similarities with previous ODEs.

This yields additional ODEs:
\begin{align}
-\frac{d\mathbf{V}}{dt} &=  \mathbf{K}^T \left( \mathbf{f}_u \bm{\gamma} + \mathbf{f}_p \right)
\label{eq:odeV} \\
-\frac{d \mathbf{Q}}{dt} &= \mathbf{K}^T \mathbf{f}_u \bm{\omega}
\label{eq:odeQ}\\
-\frac{d \mathbf{W}}{dt} &= H_{pp} + \bm{\gamma} \mathbf{H_{up}} + \mathbf{T}^T \left( \mathbf{f}_u \bm{\gamma} + \mathbf{f}_p \right) + \mathbf{f}_p \mathbf{T}
\label{eq:odeW}
\end{align}

Eventually, assuming $\mathbf{W}(t_0) < 0$, the problem parameter update is,
\begin{equation}
	d p = -\mathbf{W}(t_0)^{-1} \bm{\phi_{\nu p}}  \label{eq:update_p}
\end{equation}

Also, assuming $\mathbf{Q}(t_0)^{-1}$ exists, lagrange multipliers are updated with the linear update
\begin{equation}
	d \bm{\nu} = \epsilon_{\nu} \mathbf{Q}(t_0)^{-1} \left( \delta \bm{\phi} - \mathbf{V}(t_0) d p \right)  \label{eq:update_nu}
\end{equation}
If $\mathbf{Q}(t_0)<0$, $d \bm{\nu}$ defines a direction for the constraints reduction, and $\epsilon_{\nu} \in [0, 1]$ is thus determined to minimise the constraints $\bm{\psi}$.\par 
\vspace{\baselineskip}

The ODE system (Eqs.~(\ref{eq:odeV} - \ref{eq:odeW})) is integrated backwards from some terminal conditions depending on the terminal constraints $\bm{\psi}$ and the terminal cost $J$, while the state dynamics (Eq.~(\ref{eq:dynamic})) are integrated forward. This results in numerically decoupled schemes, which yield to numerical robustness.

\subsection{Terminal Conditions}
The formulation allows the handling of various types of constraints without changing of the backward integrated ODE, since the treatment of constraints has an influence only on the terminal conditions.\par 

The following terminal conditions are used to integrate backwards the ODEs (Eqs. (\ref{eq:odeR}), (\ref{eq:odeK}), (\ref{eq:odeT}), (\ref{eq:odeV}), (\ref{eq:odeQ}), (\ref{eq:odeW})):
\begin{align}
\mathbf{R}(t_f) &= \mathbf{J_{xx}}(t_f) + \bm{\phi_{xx}}(t_f) \label{eq:terminal_conditions_R} \\
\mathbf{K}(t_f) &= \bm{\phi_{\nu x}} (t_f) \label{eq:terminal_conditions_K} \\
\mathbf{Q}(t_f) &= 0  \label{eq:terminal_conditions_Q} \\
\mathbf{T}(t_f) &= \bm{\phi_{xp}}(t_f) \label{eq:terminal_conditions_T} \\
\mathbf{V}(t_f) &= \bm{\phi_{\nu p}} (t_f) \label{eq:terminal_conditions_V} \\
\mathbf{W}(t_f) &= \bm{\phi}_{pp} (t_f) \label{eq:terminal_conditions_W}\\
\bm{\lambda}(t_f) &= \bm{\phi_{x}}(t_f)
\end{align}

The satisfaction of the constraints $\bm{\psi}$ depends on the Lagrange multiplier $\bm{\nu}$. 

\begin{lemma}\label{lemma:constraints_reduction}
Assuming $\mathbf{H}_{uu}^{-1}$ is positive definite, the Jacobian $\bm{\psi}_{\mathbf{x}}(\mathbf{x}(t_f))$ of the constraints with respect to the terminal state $\mathbf{x}_f$ has full rank $n_{\psi}$, and $\mathbf{Q}(t_0)$ must be negative definite, then the Lagrange multipliers update $d\bm{\nu}$ gives the (local) direction of minimisation of the constraints $\bm{\psi}$.
\end{lemma}

\begin{proof}
See \citep{Olympio.Phd}. 
\end{proof}

\subsection{Parameters Determination}
The unknown parameters for the initial conditions and the dynamics do not follow the same update rule as of $dp$. Because the equations of the sensitivity matrices, ($\mathbf{R}$,  $\mathbf{K}$, and  $\mathbf{Q}$) and ($\mathbf{T}$,  $\mathbf{V}$, and  $\mathbf{W}$), are integrated backwards, setting the initial constraint $\bm{\varphi}$ does not change the terminal conditions of integration Eqs.~(\ref{eq:terminal_conditions_R}, \ref{eq:terminal_conditions_K}, \ref{eq:terminal_conditions_Q}, \ref{eq:terminal_conditions_T}, \ref{eq:terminal_conditions_V}, and \ref{eq:terminal_conditions_W}). 
Assuming the initial constraint $\bm{\varphi}$ is well defined (maximum rank condition), the update formula for the parameters $p$ is 
\begin{equation}
	d p_0 = -\bm{\varphi}_{pp}^{-1} \bm{\varphi}_p \label{eq:update_p0}
\end{equation}

There is no need to update the Lagrange vector $\bm{\eta}$ introduced in Eq.~(\ref{eq:scddevel}). One may want to include a relaxation factor to control the update $\delta p_0$.

Likewise, for the parameter part of the terminal constraint, 
\begin{equation}
	d p_f = -\bm{\psi}_{pp}^{-1} \bm{\psi}_p \label{eq:update_pf}
\end{equation}

\section{Treatment of Control Constraints} \label{sec:Control}
 
\subsection{Acceptance of Control Update}\label{sec:ControlUpdate}
Once the backward equations are integrated from their respective terminal conditions, and around a nominal trajectory, the control update is computed using Eqs.~(\ref{eq:feedback_control}) and (\ref{eq:control_update}). Note that according to Eq.~(\ref{eq:update_nu}), $\delta \mathbf{u}$ depends on $\epsilon_{\nu}$.

Both $\epsilon_{\nu}$ and $\epsilon_u$ permit of adjusting the control update to yield an improvement of the extended value function, and to ensure that the second order developments truncation error (Eq.~(\ref{eq:DeltaLagrangian})) is small enough. If Lagrange multipliers $\bm{\nu}$ are not used, a linesearch procedure is used to find $\epsilon_u$ that minimizes the merit function. When $\epsilon_{\nu}$ is used, the direction of minimization is given by a linear combination of $\bm{\alpha}$ and $d\bm{\nu}$. One approach is to perform linesearches with $\epsilon_u$ and $\epsilon_{\nu}$ successively, the purpose being to find quickly an improvement of the merit function, regardless whether it is or not an strict minimizer.

Indeed, as in \citep{mayne}, the validity of the second order development can be checked using
\begin{align}
\Delta L &= \int_{t_0}^{t_f}{ [ H(\bar{\mathbf{x}}, \bar{\bm{\lambda}}, \bar{\mathbf{u}}, p, t) - H(\bar{\mathbf{x}}, \bar{\bm{\lambda}}, \mathbf{u}, p, t) ] dt}
\label{eq:check_scndorder} \\
\Delta L &\approx \left(\Delta \mathcal{L}\right)_{d \bm{\nu} = 0}
\end{align}

Clearly,
\begin{equation}
 \lim_{\epsilon_u \rightarrow 0}{\Delta L} =  \lim_{\epsilon_u \rightarrow 0}{\left(\Delta \mathcal{L}\right)_{d \bm{\nu} = 0}} = 0
\end{equation}

For $\epsilon_u$ small enough, $\Delta L$ gets close to zero thus ensuring the second order developments are assumed valid. If in addition, the value $\Delta L$ of Eq.~(\ref{eq:check_scndorder}) is negative then the control update $d\mathbf{u}$ can be accepted.

Under these conditions the method generates iterates of control $\mathbf{u}^i(t)$ that minimize the Lagrangian $\mathcal{L}$. 

\subsection{Bang-Bang Control and Control Constraints}
According to Pontryaguin Maximum Principle\citep{Pontryaguin}, when the control belongs to a compact set $U$, for the problem class considered where the control appears linearly in the Hamiltonian, the optimal control has a bang-bang structure. Obviously, this result does not come from the variational equations (that seek $\nabla_{\mathbf{u}} H = 0$), but directly from the maximisation of the Hamiltonian over the compact set. In the problem formulation, and Eq.~(\ref{eq:control_box_constraint}), closed set on the control is assumed. However, the development of the method, referred to calculus of variations, assume the Lagrangian is differentiated on an open set.
\par 
\vspace{\baselineskip}

Constraints on the control can be set such that the control is implicitly limited. One approach is to formulate the control constraint (Eq.~	\ref{eq:control_box_constraint}) using a barrier function $\rho$ that is incorporated into the cost \citep{whiffen.2001}\citep{lantoine}. Thus, when the control violates bounds at any time, the solution gets penalised and eventually the solver corrects the violation. This approach seems to provide practical results, and also as the control bounds may not be respected in the initial iterations of the solver, the solver can go through a forbidden space of the search space and possibly  find solutions with that extra controllability. This approach, however, does not ensure the optimal control solution is an admissible control (e.g. $\mathbf{u} \notin U$).\par 

The approach considered here is to explicitly force the control to not exceed the desired limits. The principle is to identify a continuous function $m:\mathbb{R} \rightarrow U$ and to introduce a slack variable $v_i$ such that the box constraint 	Eq.~\ref{eq:control_box_constraint} can be replaced with the explicit equality constraint 
\begin{equation}
	u_i - m(v_i) = 0
\end{equation}

Indeed, one can thus replace $u_i$ accordingly in the developments, and solve the problem for $v_i$ directly instead of $u_i$. 
Since the compact set $U$ is then not used, the gradient $\nabla_v H$ can be used for the computation of the control updates, and this also provides an indication of the optimality of the control $\mathbf{u}$ (Weierstrass condition). Such an application $m$ can be for instance,
\begin{equation*}
	m(v_i) = u_i^l + (u_i^u-u_i^l) \sin^2{v_i} 
\end{equation*}

This approach, although numerically more challenging than penalization, makes sure that the control satisfies the problem description and never violates the box constraints Eq.~(\ref{eq:control_box_constraint}).

\subsection{Heuristic for Convergence of Bang-Bang Problems}
Some transformations $m:\mathbb{R} \rightarrow U$ may allow, however, triviality or singularity to the problem. For instance, when Lagrange multipliers $\bm{\nu}$ are not used, for the i-th component of the control $\mathbf{u}$, some transformations $m$ yield,
\begin{equation}
	\exists t \in [t_0, t_f], m(v_i(t)) = 0 \Rightarrow ( H_{v_i}(t) = 0, \delta v_i(t) = 0)
\end{equation}
No updates are made even though the control solution is not optimal.\par 
A diagonal matrix $\mathbf{D} > 0$ is thus introduced in the Hamiltonian,
\begin{equation}
	H(\mathbf{x}, \bm{\lambda}, \mathbf{u}, p; t) = \bm{\lambda}^T \mathbf{f}(\mathbf{x}, \mathbf{u}, p; t) + \mathbf{u}^T \mathbf{D} \mathbf{u}
\end{equation}
This is equivalent to adding a quadratic Lagrange term in the cost $J = J(\mathbf{x}(t_f), \mathbf{u}, p_f)$. As a result, control $\mathbf{u}$ is smoothed. To obtain a bang-bang saturated control, $\mathbf{D}$ is updated and decreased periodically with the iterations for the control to be in a neighbourhood of an optimal bang-bang control solution.


\section{Continuous Solution Approximation} \label{sec:mesh}

\subsection{Problematic} 
To compute the feedback control coefficients, $\bm{\alpha}$, $\bm{\beta}$, $\bm{\gamma}$ and $\bm{\omega}$, a backwards integration (e.g. Eqs.~(\ref{eq:odeR}), (\ref{eq:odeK}), \ref{eq:odeT}) is done around a nominal trajectory $\{\mathbf{x}, p\}$ that results from a forward integration. 
It is thus not possible to solve all the ODEs concurrently. It is necessary to store continuous solution of the ODEs, that is both the trajectory (e.g. $\mathbf{x}$ and $\bm{\lambda}$), and the sensitivity matrices (e.g. $\mathbf{R}$, $\mathbf{K}$) to then compute the control update. To respect the continuous developments of previous section, an efficient and reliable continuous data storage mechanism has to be devised.\par 

\subsection{Current Discrete Approximation Approaches} 
Often, the alternative is to work with the best discrete time approximate solutions and the definition of a time mesh where the solution can be evaluated.

Many approaches have been proposed in the literature to produce a discrete control time mesh. The approach of keeping the control constant for fixed durations \citep{lantoine}\citep{whiffen.2001} produces good results when the number of nodes has been well selected. Likewise, the control can be kept constant over multiple Runge-Kutta time steps for smooth control\citep{Colombo}. But, since the control is used in the forward integration but not during the backwards integration, using the Runge-Kutta (RK) time steps alone does not seem appropriate for the bang-bang problem considered. Runge-Kutta integration is only time-reversal when used with a constant step size \citep{Stoffer}. 
Others\citep{varin}\citep{Betts.98} propose a refinement procedure based on an integer programming problem, and consider the minimisation of the continuity error and the minimisation of number of the points to add. \par 
When the control is of a bang-bang structure, an other approach should be followed, because a constant time mesh for the control is likely to miss the switching times of the optimal control. The usual approaches increase the number of points close to the discontinuity. This has a major drawback of increasing the number of data points to store. 

The choice of a discrete, or multi-stage, formulation could be argued in regards to potential applications, since it is always simpler to implement a piecewise control. From a theoretical point of view, however, it is necessary to propose a scheme for limiting the approximation made on the control. Properties of the continuous control solution (e.g. switching structure, singular arcs) can eventually be used to produce a good discrete control solution.  

\subsection{Continuous Trajectory Approximation} \label{sec:ContinuousApprox}
The trajectory and the control are initially sampled at $N$ given node points that can be described by a grid $t_i \in [t_0, t_f]$, $i=1..N$. The trajectory and the control are thus exactly defined at the nodes, and an approximation to the continuous solution is constructed between the node points.
Consider an explicit adaptive step-size Runge-Kutta integration method. It is the baseline for solving the ODE systems in the current paper. An explicit Runge-Kutta integration method uses the following iterative equations for integrating, 
\begin{align}
	\mathbf{y}_{n+1} &= \mathbf{y}_n + h\sum_{i=1}^s b_i \mathbf{k}_i \label{eq:RK}\\
	\mathbf{k}_i &= \mathbf{f}\left(t_n + c_i h, \mathbf{y}_n + h \sum_{j = 1}^s a_{ij} \mathbf{k}_j\right),
\end{align}
where $a_{ij}$, $b_i$ and $c_i$ are constant scalars, and $s$ is the order of the method. 
Shampine's method for dense output computation in Runge-Kutta computation is used\citep{shampine}. One can introduce a polynomial $\mathbf{S}(x)$, and an error function $\mathbf{e}(x)$ such that,
\begin{align}
\mathbf{y}(x) &= \mathbf{S}(x) + \mathbf{e}(x) \\
\frac{d \mathbf{e}}{dx} &= \mathbf{f}(x, \mathbf{S}(x)) - \frac{\mathbf{S}(x)}{dx} \qquad \mathbf{e}(x_0) = 0 \nonumber
\end{align}

Shampine\citep{shampine} thus defines polynomials $\mathbf{P}(x)$ of maximum degree $m+1$ ($m$ being the number of sub-integrations in one subinterval of length $h$), 
\begin{equation}
\frac{d\mathbf{P}(x)}{dx} = \sum_{i=0}^{m}{l_{i,m}(x) \mathbf{f}(x_i, \mathbf{e}_i+\mathbf{S}(x_i))}
\end{equation}
where $l_{i,m}$ are Lagrange basis polynomials of kth-order $m$. It is demonstrated there exists an integer $n$, and constants $c_1$ and $c_{2,m}$ such that
\begin{align}
\norm{\frac{d^2 \mathbf{y}(x)}{d^2x} - \frac{ d\mathbf{S} (x)}{d^2x}} &\leq c_{1} h^n\\
\norm{\mathbf{y}(x) - \mathbf{P}(x)} &\leq c_{2,m} h^{2+\min{(m, n+1)}}
\end{align}
That is, $\mathbf{P}$ is a better higher order approximation of $\mathbf{y}$ that $\mathbf{S}$ is.
An iterative scheme is used to construct $\mathbf{P}(x)$. The procedure starts with a simple Euler integration model for $\mathbf{S}(x)$ with $m=1$, and then increases the order $m$ by evaluating successively $\mathbf{S}(x)$ and $\mathbf{P}(x)$. The construction of $\mathbf{S}$ and $\mathbf{P}$ is dependant of the extrapolation method (e.g. Runge-Kutta method), and $\mathbf{P}$ is only then available at the end of the integration over the interval of length $h$.\par 
For a fifth order explicit RK method, $\forall x \in [x_n, x_{n+1}]$, 
\begin{equation}
	\mathbf{y}(\bar{x}) = \mathbf{r}_1 + \bar{x} (\mathbf{r}_2 + (1-\bar{x})(\mathbf{r}_3 + \bar{x} (\mathbf{r}_4 + (1-\bar{x})\mathbf{r}_5))),
	\label{eq:continuous_approx}
\end{equation}
where
\begin{align}
\mathbf{r}_1 &= \mathbf{y}_n\\
\mathbf{r}_2 &= h\sum_{i=1}^s b_i \mathbf{k}_i\\ 
\mathbf{r}_3 &= h\mathbf{k}_1 - (\mathbf{y}_{n+1} - \mathbf{y}_n)\\
\mathbf{r}_4 &= -h (\mathbf{k}_7 + \mathbf{k}_1) \\
\mathbf{r}_5 &= h  (d_1\mathbf{k}_1+d_3\mathbf{k}_3+d_4\mathbf{k}_4+d_5\mathbf{k}_5 +d_6\mathbf{k}_6+d_7\mathbf{k}_7)
\end{align}
and $d_i$ are constants of the method. Globally, the solution for the state trajectory and the sensitivity matrices can be stored as a $\mathcal{C}^1$ approximation. For a fifth order RK method, five coefficients for each element of the ODE system, plus the range of validity of the polynomials are stored. \par 
Consequently, during the forward integration, the trajectory state $\mathbf{x}(t)$ is stored continuously, using Eq.~(\ref{eq:continuous_approx}), and this continuous approximation is used during the backward integration. Likewise, during the backward integration, the costate $\bm{\lambda}(t)$ and the sensitivity matrices $\mathbf{R}(t), \mathbf{K}(t), \mathbf{T}(t), \mathbf{V}(t)$ continuous approximation are stored. Thus, all matrices and vectors involved in the computations follow a high resolution continuous approximation. Eventually, the backward stage being fully continuous, a continuous control update $\delta \mathbf{u}$ is available.\par 
Additionally, to obtain the best trajectory and the most accurate optimal control, from the continuous control update $\delta \mathbf{u}$, an approach should be devised to construct a continuous approximation of the nominal control $\bar{\mathbf{u}}$.

\subsection{Continuous Nominal Control Approximation}
The continuous control approximation is constructed a-posteriori, from the results of the backward integration. In regards to Eq.~(\ref{eq:control_update}), the continuous control update should be added to the nominal control $\bar{\mathbf{u}}$ to provide an updated control $\mathbf{u}$ that can be used as the new nominal control for the following iteration during the backward integration. 
Consider a nominal control $\bar{\mathbf{u}}$, during the forward recursion, when integrating ODE Eq.~(\ref{eq:dynamic}). According to section~\ref{sec:ContinuousApprox}, control update $\delta \mathbf{u}$ can be constructed accurately using RK interpolant of matrices $\mathbf{R}$, $\mathbf{K}$ for all $t \in [t_0, t_f]$, and likewise, an accurate updated control $\mathbf{u}$ can be computed. This online produced nominal control has however to be stored, such that it can be used again for the backward recursion and future iterations. Practically, this is a difficult problem since neither $\bar{\mathbf{u}}$ nor $\delta \mathbf{u}$ have a readily available closed form, in general.\par  
Consider $t_k$ the RK nodes of the forward integration with the control $\bar{\mathbf{u}} + \epsilon_u\delta \mathbf{u}$ with values $\mathbf{u}(t_k)$.
The problem is to completely determined $\mathbf{u}$ such that
\begin{align}
|J(\bar{\mathbf{u}}+\epsilon_u\delta\mathbf{u}) - J(\mathbf{u})| &\leq \eta_1 \\
\norm{\phi(\bar{\mathbf{u}}+\epsilon_u\delta\mathbf{u}) - \phi(\mathbf{u})} &\leq \eta_2
\end{align}
where $J$ and $\phi$ are, respectively, the target merit function value and the constraints, computed with the improved control $\bar{\mathbf{u}} + \delta \mathbf{u}$. $\eta_1$ and $\eta_2$ are given small tolerances.
For very sensitive problems, the discrepancy between this new control $\mathbf{u}$ and the control $\bar{\mathbf{u}}+\epsilon_u\delta \mathbf{u}$ can result in large difference in the merit function value. This difference can possibly be bigger than the expected improvement (Eq.~(\ref{eq:check_scndorder}) ) thus resulting in a failure of the method to converge properly.\par

Each component of $\mathbf{u}$ is thus approximated between the coarse grid nodes to follow a piecewise-polynomial function, or spline (linear function, cubic spline interpolant). Intermediate nodes at $t_{k+1/2}$ can be tested, and added if it results in a better approximation of $\bar{\mathbf{u}} + \epsilon_u \delta \mathbf{u}$. Eventually, $\mathbf{u} \rightarrow \bar{\mathbf{u}}$.

\section{Second Order Algorithm and Implementation} \label{sec:algo}

\subsection{The Algorithm}

The algorithm is implemented (C-SOGO) and coded in C++. It uses efficient linear algebra classes with shared-memory parallelisation. 
A high level description of the algorithm is the following: 


Given:
\begin{itemize}
	\item convergence tolerance on the constraints $\eta_{\psi}$, and the optimality condition $\eta_{H}$.
	\item Trust-Region radius for the control $\Delta_u > 0$.
	\item Trust-Region radius for Lagrange multipliers $\Delta_{\nu} > 0$, and parameter  $\Delta_{p} > 0$.
	\item a nominal control $\bar{\mathbf{u}}(t)$ for $t \in [t_0, t_f]$.
	\item a nominal Lagrange multiplier $\bar{\bm{\nu}} \in \mathbb{R}^{n_{\psi}}$.
	\item a regularisation matrix $\mathbf{C} \in \mathbb{R}^{n_{\psi}, n_{\psi}}$.
	\item initial state $\mathbf{x}_0 = \mathbf{x}(t_0)$.
\end{itemize}

\begin{description}
\item[Step 0]  \textbf{Initialisation}
\begin{description}
	\item[Step 0.a.] Initialise iteration counter.
	\item[Step 0.b.] Set $\epsilon^*_u=0$ and $\epsilon^*_{\nu}=0$.
\end{description}
\par
\vspace{0.7\baselineskip}

\item[Step 1]  \textbf{Nominal trajectory}
\begin{description}
	\item[Step 1.a.] Using a nominal control $\bar{\mathbf{u}}(t)$, compute the state trajectory $\mathbf{x}(t)$. 
	\item[Step 1.b.] Compute nominal merit function value $\bar{L}$.
\end{description}
\par
\vspace{0.7\baselineskip}

\item[Step 2.] \textbf{Check terminal criteria}
\begin{description}
	\item Compute the terminal constraints $\bm{\psi}$ Eq.~(\ref{eq:terminal_constraint}). If constraints $\norm{\psi} > \eta_{\psi}$ go to Step 3.
	\item If optimality condition $\max{\left(\Delta L, \norm{H_u}_{\infty}\right)} > \eta_{H}$ go to Step 3.
	\item Otherwise, an optimal solution has been found. Stop.
\end{description}
\par
\vspace{0.7\baselineskip}

\item[Step 3]  \textbf{Backwards Step}
\begin{description}
	\item[Step 3.a.]  Computation of the terminal conditions with Eqs.~\ref{eq:terminal_conditions_K}, ~\ref{eq:terminal_conditions_Q}, ~\ref{eq:terminal_conditions_R} and Eqs.~\ref{eq:terminal_conditions_T}, ~\ref{eq:terminal_conditions_W}, ~\ref{eq:terminal_conditions_K}. 
	\item[Step 3.b.] Computation of the Shift matrix $\mathbf{S_{uu}}$ thanks to the Trust-Region algorithm (see definition~\ref{def:trustregion}) for $\mathbf{H_{uu}}(t) + \mathbf{S_{uu}}(t) > 0$ for all $t \in [t_0, t_f]$. The Shift matrix $\mathbf{S_{uu}}$ shall ensure the boundedness of the ODE solutions.
	\item[Step 3.c.] Integration of Eqs.~(\ref{eq:odeR}, \ref{eq:odeK}, and \ref{eq:odeQ}) and computation of continuous time approximation of $\bm{\lambda}$(t), $\bm{R}(t)$, $\bm{K}(t)$ and $\bm{Q}(t)$ for all $t$. During the integration, $\mathbf{H}_{uu}$ is shifted with $\mathbf{S}_{uu}$. As the equations are integrated, in accordance to sec.~\ref{sec:ContinuousApprox}, a mesh and polynomials are constructed to provide a continuous approximation of them.
	\item[Step 3.d.] Likewise, if the problem includes a parameter, integrate Eqs.~(\ref{eq:odeT}), (\ref{eq:odeV}), and (\ref{eq:odeW}) and compute continuous time approximations of $\bm{T}(t)$, $\bm{V}(t)$ and $\bm{W}(t)$ for all $t$.
\end{description}
\par
\vspace{0.7\baselineskip}

\item[Step 4]  \textbf{Computation of Lagrange multipliers $\bm{\nu}$}
\begin{description}
	\item[Step 4.a.] If $\mathbf{Q}(t_0) < 0$ go to 4.c
	\item[Step 4.b.] Apply Trust-Region algorithm to have a negative definite $\mathbf{Q}(t_0)$.
	\item[Step 4.c.] Computation of Lagrange multipliers $\bm{\nu}$ for the constraints using Eq.~(\ref{eq:update_nu}).
\end{description}
\par
\vspace{0.7\baselineskip}

\item[Step 5.]  \textbf{Forward Integration Loop}
\begin{description}
	\item[Step 5.a.] Integrate the dynamical state equations Eq.~(\ref{eq:dynamic}) with the control $\mathbf{u} + \epsilon_u \bm{\delta} \mathbf{u}$, where the control update $\bm{\delta} \mathbf{u}$ is computed using Eq.~(\ref{eq:control_update}). The continuous approximation of sensitivity matrices $\mathbf{R}(t)$, $\mathbf{K}(t)$, and $\mathbf{Q}(t)$ accurately provide the feedback coefficients $\alpha$, $\beta$ and $\gamma$. To produce the improving control update $\bm{\delta} \mathbf{u}$ algorithm A1 is used to find $\epsilon_u$ and $\epsilon_{\nu}$ that minimise the merit function Eq.~(\ref{eq:value_function}).
	\item[Step 5.b.] Evaluation of the constraints $\bm{\psi}(\mathbf{x}; t_f)$, the objective function $J(\mathbf{x}; t_f)$, and the merit function value $L(\mathbf{x}, \bm{\nu}; t_f)$.
	\item[Step 5.c.] If improvement in merit function, $L<\bar{L}$, take $\epsilon^*_u=\epsilon_u$ and $\epsilon^*_{\nu}=\epsilon_{\nu}$. Go to step 6.
	\item[Step 5.d.] Improvement could not be found. Conditions of validity of the second order development might not be meet. Reduce Trust-Region radius $\Delta_u$. Increase Iteration counter. Go to step 3.
\end{description}
\par
\vspace{0.7\baselineskip}

\item[Step 6.]  \textbf{Accept current control.} 
\begin{description}
	\item[Step 6.a.] Constructing a continuous time approximation of the new nominal control $\bar{\mathbf{u}} \leftarrow \bar{\mathbf{u}} + \epsilon^*_u \bm{\delta} \mathbf{u}$. 
	\item[Step 6.b.] Update Lagrange multipliers, $\bm{\nu} = \bar{\bm{\nu}} + \epsilon^*_{\nu} d\bm{\nu}$.
	\item[Step 6.c.] Update nominal merit function value $\bar{L}$.
	\item[Step 6.d.] Update regularisation matrix $\mathbf{C}$.
\end{description}
\par
\vspace{0.7\baselineskip}

\item[Step 7.]  \textbf{Continue with the next iteration.} Increase iteration counter. Go to step 2.

\end{description}

\par
\vspace{0.7\baselineskip}

Algorithm A1 is either a line-search method to get $\epsilon_u \in [0,1]$ that minimises $\mathcal{L}$ when $\bm{\nu}$ is not used, or otherwise, a simple two-dimensional $[\epsilon_u, \epsilon_{\nu}]$-grid search  method to get the additional parameter $\epsilon_{\nu} \in [0,1]$. On the current iteration, linesearch, or gridsearch, is stopped when a sufficient decrease is obtained.

Convergence is achieved when both the constraints $\bm{\psi}$ and the optimality norm are satisfied to a given tolerance. A norm of optimality can be given by Eq.~(\ref{eq:check_scndorder}) and $\mathbf{H_u}$.

\section{Examples}\label{sec:examples}
The initial motivation of this work has been the optimization of low-thrust space trajectories, which present a bang-bang optimal control structure.
Both presented examples have affine control dynamics with the control defined on a closed set, and thus the optimal control should be bang-bang.

\subsection{Double Integrator Problem}\label{sec:example1}

The dynamics
\begin{equation}
		\mathbf{f}(\mathbf{x}, \mathbf{u}; t) = \frac{d }{dt}
		\begin{bmatrix}
			p\\
			v
		\end{bmatrix}
		=
			\begin{bmatrix}
				 v\\
				 u
			\end{bmatrix}
	\label{eq:ex_dynamic}
\end{equation}
where $p$ is the 1-D position of the particle, and $v$ is its velocity.
The control bounds are
\begin{equation}
	-1 \leq u \leq 1
\end{equation}
Owing to the control constraints, the transformation $u = 1-2\sin^2 v$ is used, and $v$ is the new control to seek.

Terminal constraints are
\begin{equation}
 \bm{\phi}(\mathbf{x}; t_f) = 
		\begin{bmatrix}
			p(t_f) = 0\\
			v(t_f) = 0
		\end{bmatrix}
\end{equation}
And the initial conditions are $p(0)=1$, $v(0)=1$. 

The objective is to minimize the total time,.
\begin{equation}
	J(\mathbf{x}; t_f) = t_f    \rightarrow \max
\end{equation}
The time of flight $t_f$ is thus considered as a parameter of the problem. The integration of the equations is thus done through a change of variable $\tau = t/t_f$ such that the integration is done on $[0, 1]$ with varying jacobian $d \tau$.
That is we should use
\begin{equation}
	\mathbf{\tilde{f}}(\mathbf{x}, \mathbf{u}, t_f; \tau) = \mathbf{f}(\mathbf{x}, \mathbf{u}; \tau t_f) t_f
\end{equation}

\begin{figure}[ht]
\centering
\includegraphics[width=1.\textwidth]{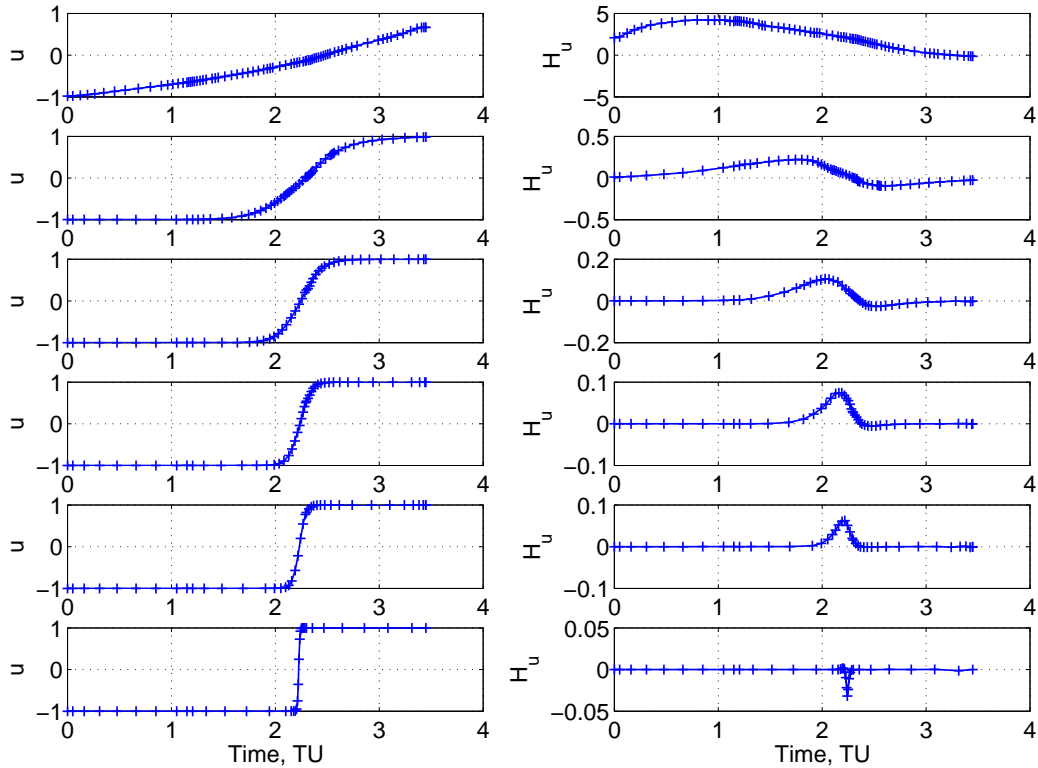}
\caption{Iterates of a simple control problem. The coarse grid includes 3 node points. Left, continuous control iterates. Right, gradient of the Hamiltonian with respect to the continuous control.}
\label{fig:adaptive_mesh}
\end{figure}

Figure~\ref{fig:adaptive_mesh} shows the iterates of program to find the solution. The control is approximated with a constrained cubic spline to prevent overshoot at intermediate points (with constrained cubic splines, the slope is imposed rather than the curvature). At the plateau of the control, in figure~\ref{fig:adaptive_mesh}, the control is thus approximated with straight lines.\par 
Using as initial guess $t_f = 3.5$ we found an optimal time $t_f^* = 3.449581$.
\par 
One can note how the number of points varies. The comparison  to the smoothness of the control or the optimality of the solution is not straightforward, however.\par 
It is arguable that the control obtained with the method is not strictly bang-bang in the sense that it can reach intermediate values on a set of measure not null. However, from a numerical point of view, for an optimality tolerance small enough the error to the true bang-bang solution should be negligible. And, possibly, the solution obtained can be refined in a later stage with the switching time as parameter since an optimal control structure is then known. The result of the refinement is likely to be negligible if the optima

\subsection{Simple Orbital Transfer Problem}\label{sec:example3}

The problem is the one of transferring a spacecraft from one circular orbit to an other with higher radius. The terminal orbit is sought in the same plane as the initial orbit, and thus the dynamics is simplified to be two-dimensional. The spacecraft is propelled with continuous and constant thrust. With the gravitational constant $\mu$, the dynamics are then,
\begin{equation}
		\mathbf{f}(\mathbf{x}, \mathbf{u}; t) = \frac{d }{dt}
		\begin{bmatrix}
			r\\
			v_r\\
			v_{\theta}\\
			m
		\end{bmatrix}
		=
			\begin{bmatrix}
			v_r\\
			\frac{v_{\theta}^2}{r} - \frac{\mu}{r^2} + \frac{F}{m} \sin u\\
			-\frac{v_r v_{\theta}}{r} + \frac{F}{m} \delta \cos u\\
			-\frac{F}{V_e}
			\end{bmatrix}
	\label{eq:ex_dynamic_orbital}
\end{equation}

State elements $r$, $v_r$, $v_{\theta}$ and $m$ denote respectively the radial position, the radial velocity, the ortho-radial velocity and the mass of the spacecraft. $F$ is the thrust force ($0 \leq F \leq F_0 = 5 N$) and $V_e$ is the exhaust velocity ($V_e = g_0 I_{sp}= 19612 m/s$). The controls are thus the steering angle $u$ and the thrust amplitude $F$. 

The initial orbit is circular of radius $r_0=20000$ km. The constraint is to be on a circular orbit of radius $r_f=42000$ km at final time. Using the cylindrical coordinate, this yields the constraints,
\begin{equation}
 \bm{\phi}(\mathbf{x}; t_f) = 
		\begin{bmatrix}
			v_r\\
			v_{\theta} - \sqrt{\frac{\mu}{r}}\\
			(r - r_f)^2
		\end{bmatrix}
	\label{eq:ex_constraint_orbital}
\end{equation}

The time of flight is fixed to $t_f = 4$ days. The initial mass is $1000$ kg. The objective is to maximize the final mass.
\begin{equation}
	J(\mathbf{x}; t_f) = m(t_f)  \rightarrow \max
\end{equation}

The final optimal mass is $m(t_f)^* = 932.15$ kg and the spacecraft make about $8$ revolutions around Earth.
The optimal control is depicted on figure~\ref{fig:ex_ctl_orbital}, and the optimal trajectory is on figure~\ref{fig:ex_traj_orbital}. As can be seen on figure~\ref{fig:ex_ctl_orbital}, the optimal control is of bang-bang type. The solver is able to find accurately the commutation points, while no insight of the control structure is initially provided. 
Such problem has usually many local optimal that can be sorted with respect to the number of revolutions. That is some solutions may have more, or less, than 8 revolutions for a constant time of flight. In general however, the best control is to thrust at the perigee (where the cost and the terminal constraints are more sensible to thrust) and coast at the apogee, as depicted in figure~\ref{fig:ex_traj_orbital}. 

\begin{figure}[ht]
\centering
\includegraphics[width=0.8\textwidth]{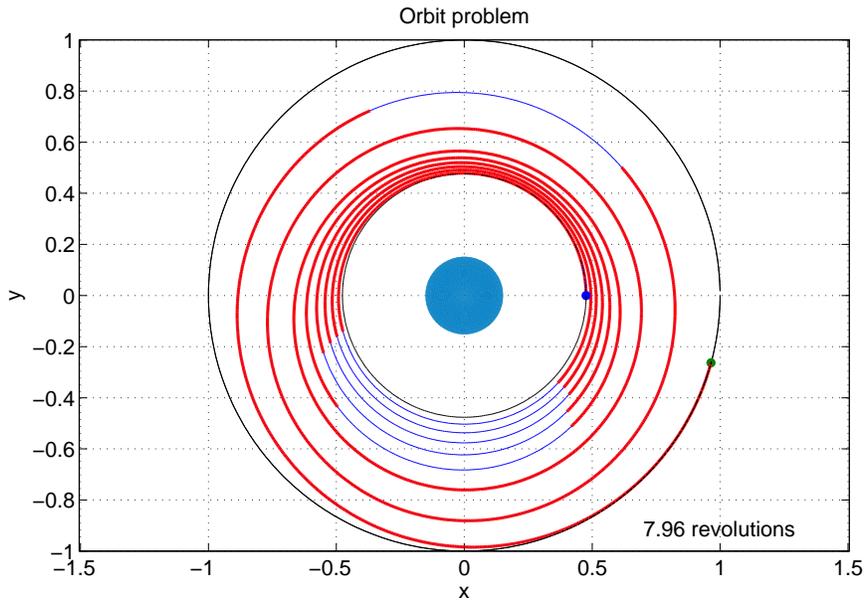}
\caption{Optimal trajectory of the orbital transfer problem.}
\label{fig:ex_traj_orbital}
\end{figure}

\begin{figure}[ht]
\centering
\includegraphics[width=0.7\textwidth]{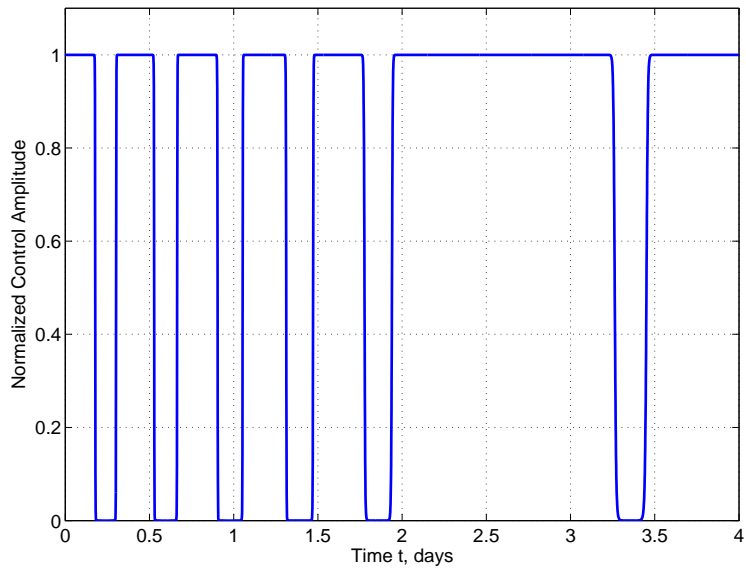}
\caption{Control of the orbital transfer problem.}
\label{fig:ex_ctl_orbital}
\end{figure}

\clearpage

\section{Conclusions} 

The solving of optimal control problems with a continuous backward-forward sweep algorithm based on a second-order variations is presented. A second order expansion of a Lagrangian provides linear updates of the control to construct a locally feedback optimal control of the problem. 
The control updates are computed following backward and forward stages. Ordinary differential equations are integrated backward around a state trajectory that has been computed forwardly in an earlier stage. The method uses an accurate continuous approximation of the ODE solutions to ensure precision of the control updates. \par
The method was demonstrated on two examples with bang-bang optimal control solution. It was shown that the solver can find accurately the switching structure of the solution, without any insight.


\bibliographystyle{unsrt}
\bibliography{paper}

\end{document}